\documentclass[12pt]{amsart}

\usepackage{amssymb}

\usepackage{enumerate}

\usepackage{graphicx}

\makeatletter
\@namedef{subjclassname@2010}{%
  \textup{2010} Mathematics Subject Classification}
\makeatother

\newtheorem{thm}{Theorem}[section]
\newtheorem{cor}[thm]{Corollary}
\newtheorem{lem}[thm]{Lemma}
\newtheorem{prop}[thm]{Proposition}

\theoremstyle{definition}
\newtheorem{defin}[thm]{Definition}
\newtheorem{rem}[thm]{Remark}

\numberwithin{equation}{section}

\def\BMO{\mathrm{BMO}}
\def\LMO{\mathrm{LMO}}
\def\bmo{\mathrm{bmo}}
\def\lmo{\mathrm{lmo}}

\newcommand{\bprop} {\begin{proposition}}
\newcommand{\eprop} {\end{proposition}}
\newcommand{\btheo} {\begin{theorem}}
\newcommand{\etheo} {\end{theorem}}
\newcommand{\blem} {\begin{lemma}}
\newcommand{\elem} {\end{lemma}}
\newcommand{\bcor} {\begin{corollary}}
\newcommand{\ecor} {\end{corollary}}

\newcommand{\Be}{\begin{equation}}
\newcommand{\Ee}{\end{equation}}
\newcommand{\Bea}{\begin{eqnarray}}
\newcommand{\Eea}{\end{eqnarray}}
\newcommand{\Bes}{\begin{equation*}}
\newcommand{\Ees}{\end{equation*}}
\newcommand{\Beas}{\begin{eqnarray*}}
\newcommand{\Eeas}{\end{eqnarray*}}
\newcommand{\Ba}{\begin{array}}
\newcommand{\Ea}{\end{array}}

\def\T{\mathbb{T}}

\begin{document}


\baselineskip=17pt



\title[Bounded mean oscillation on rectangles.]{An embedding relation for bounded mean oscillation on rectangles.}

\author[Beno\^it F. Sehba]{Beno\^it F. Sehba}
\email{bsehba@gmail.com}


\date{}

\begin{abstract}
In the two-parameter setting, we say a function belongs to the mean little $\BMO$, if its mean over any interval and with respect to any of the two variables has uniformly bounded mean oscillation. This space has been recently introduced by S. Pott and the author in relation with the multiplier algebra of the product $\BMO$ of Chang-Fefferman. We prove that the Cotlar-Sadosky space of functions of bounded mean oscillation $\bmo(\T^N)$ is a strict subspace of the mean little $\BMO$.
\end{abstract}



\maketitle

\section{Introduction and results}
\subsection{Introduction}

 In the two-parameter case, the mean little $\BMO$ space consists of those functions such that their mean over any interval with respect to any of the two variables is uniformly in $\BMO(\T)$.  This space was introduced recently in the literature by S. Pott and the author in their way to the characterization of the multiplier algebra of the product $\BMO$ of Chang-Fefferman (\cite{ChFef1, pseh2, benoit}). Its definition is very close in spirit to the one of the little $\BMO$ of Cotlar and Sadosky (\cite{cotsad}) and this is somehow misleading. It is pretty clear that the little $\BMO$ embeds continuously into the mean little $\BMO$ and it was natural to ask if both spaces are the same. To find out, we use an indirect method; we characterize  the multiplier algebra of the Cotlar-Sadosky space and the set of multipliers from the little $BMO$ to the mean little $BMO$.

\subsection{Definitions and results}
Given two Banach spaces of functions $X$ and $Y$, the space of pointwise multipliers from $X$ to $Y$ is defined as follows
$$\mathcal {M}(X,Y)=\{\phi:\phi f\in Y\,\,\,\textrm{for all}\,\,\,f\in X\}.$$
When $X=Y$, we simply write $\mathcal {M}(X,X)=\mathcal {M}(X)$.

The so-called small BMO space on $\T^N$, introduced by Cotlar and Sadosky and denoted $\bmo(\T^N)$ consists of functions $b \in L^2(\T^N)$ such that the quantity
\begin{equation} \label{bmo1}
   ||b||_{*,N}:=\sup_{R \subset \T^N, \text{ rectangle }} \frac{1}{|R|}\int_R |b(t_1,\cdots,t_N) - m_R b| dt_1\cdots dt_N 
\end{equation}
is finite,
$m_Rb=\frac{1}{|R|}\int_Rb(t_1,\cdots,t_N)dt_1\cdots dt_N$. Seen as a quotient space with the set of constants, $\bmo(\T^N)$ is a Banach space with norm  $||b||_{\bmo(\T^N)}:=||b||_{*,N}.$


Note that in the above definition, since $R$ is a rectangle in $\mathbb {T}^N$, $m_Rf$ is a constant. We will sometimes consider the case where $R$ is a rectangle in $\mathbb {T}^M$ with $M$ an integer, $0<M<N$, in which case $m_Rf$ is a function of $N-M$ variables.

Another notion of function of bounded mean oscillation was introduced in \cite{benoit} in the two-parameter setting. This notion is inspired from the one of M. Cotlar and C. Sadosky (\cite{cotsad}). One of its higher-parameter versions is defined as follows.
\begin{defin}\label{meanbmo}
A function $b\in L^2(\T^N)$ belongs to $\bmo_{m}(\T^N)$ if there is a constant $C>0$ such that for any integers $0<N_1,N_2<N$, $N_1+N_2=N$ and any rectangle $R\subset \T^{N_1}$,
$$||m_R b||_{*,N_2}\le C.$$
\end{defin}
 The space $\bmo_m(\T^N)$ seen as a quotient space by the set of constants is a Banach space under the norm  $$||b||_{\bmo_{m}(\T^N)}:=C^*$$ where $C^*$ stands for the smallest constant in the above definition.

It is clear from the definitions above that $\bmo(\T^N)$ embeds continuously into $\bmo_m(\T^N)$.
We will be calling $\bmo_m$ the mean little $\BMO$.  Our main result is the following.
\begin{thm}\label{mainembed}
$\bmo(\T^N)$ is strictly continuously embedded into $\bmo_m(\T^N)$.
\end{thm}
To prove the above theorem, we first prove the following.
\begin{thm}\label{bmomultconst}
The only pointwise multipliers of $\bmo(\T^N)$ are the constants.
\end{thm}

We say a function $b\in L^2(\T^N)$ has bounded logarithmic mean oscillation on rectangles, i.e $b\in lmo(\T^N)$ if
\Beas ||b||_{*,\log,N} &:=& \sup_{R=I_1\times\cdots \times I_N\subset \T^N}\frac{\sum_{j=1}^N\log\frac{4}{|I_j|}}{|R|}\int_{R}|b(t)-m_Rb|dt\\ &<& \infty.
\Eeas
Let us introduce also the mean little $\LMO$ space in product domains.
\begin{defin}\label{meanlmo}
A function $b\in L^2(\T^N)$ belongs to $\lmo_{m}(\T^N)$ if there is a constant $C>0$ such that for any decomposition $0<N_1,N_2<N$, $N_1+N_2=N$, and any rectangle $R\subset \T^{N_1}$,
$$||m_R b||_{*,\log,N_2}\le C.$$
\end{defin}
If $C^*$ stands for the smallest constant in the Definition \ref{meanlmo}, then seen as a quotient space by the set of constants, $\lmo_m(\T^N)$ is a Banach space with the following norm
  $$||b||_{\lmo_{m}(\T^N)}:=C^*.$$

In terms of multipliers,  to get close to the one parameter situation, we need to start from $\bmo(\T^N)$ and take $\bmo_m(\T^N)$ as the target space.
\begin{thm}\label{meanbmomulti}
Let $\phi\in L^2(\T^N)$. Then the following assertions are equivalent.
\begin{itemize}
\item[(i)] $\phi$ is a multiplier from $\bmo(\T^N)$ to $\bmo_m(\T^N)$.
\item[(ii)] $\phi\in \lmo_m(\T^N)\cap L^\infty(\T^N)$.
\end{itemize}
Moreover, $$\|M_\phi\|_{\bmo(\T^N)\rightarrow \bmo_m(\T^N)}\simeq \|\phi\|_{L^\infty(\T^N)}+\|\phi\|_{\lmo_m(\T^N)}$$
where $\|M_\phi\|_{\bmo(\T^N)\rightarrow \bmo_m(\T^N)}$ is the norm of the multiplication operator from $\bmo(\T^N)$ to $\bmo_m(\T^N)$.
\end{thm}

Theorem \ref{bmomultconst} and Theorem \ref{meanbmomulti} clearly establish Theorem \ref{mainembed} since\\ $\lmo_m(\T^N)\cap L^\infty(\T^N)$ contains more than constants. The proofs are given in the next section.
The last section of this note also states that the only multiplier from a Banach space of functions (strictly) containing $\bmo(\T^N)$ to $\bmo(\T^N)$ is the constant zero.

As we are dealing only with little spaces of functions of bounded mean oscillation, we essentially make use of the one parameter techniques. This is not longer possible when considering the multipliers of the product $\BMO$ of Chang-Fefferman for which one needs more demanding techniques (\cite{pseh1, pseh2, benoit}).
\section{Comparison via multiplier algebras }
\subsection{Proof of Theorem \ref{bmomultconst}}
The space $\bmo(\T^N)$ has the following equivalent definitions (\cite{cotsad, fergsad}) that we need here.
\begin{prop}\label{prop:equivdefbmo}
The following assertions are equivalent.
\begin{itemize}
\item[(1)] $b\in \bmo(\T^N)$.
\item[(2)] $b \in L^2(\T^N)$ and there exists a constant $C>0$ such that for any decomposition $N_1+N_2=N$, $0<N_1,N_2<N$,
\begin{itemize}
\item[(i)] $\|b(\cdot, t)\|_{* ,N_1}\le C$,  for all $t\in \T^{N_2}$.
\item[(ii)] $ \|b(s, \cdot)\|_{*,N_2 }     \le C$ for all $s \in \T^{N_1}$ .
 \end{itemize}

\end{itemize}
\end{prop}
\begin{proof}
The proof was given in the two-parameter case in \cite{cotsad}. It is essentially the same proof in the multi-parameter setting. We follow the simplified two-parameter proof from \cite{brett}.

We first suppose that $b\in \bmo(\T^N)$ that is we have that for any $S\subset \T^{N_1}$, $K\subset \T^{N_2}$,  $N_1+N_2=N$,
$$\frac{1}{|S|}\frac{1}{|K|}\int_{S}\int_{K}|b(s,t)-m_{S\times K}b|dsdt\le \|b\|_{*,N}.$$
If $S=S_1\times\cdots\times S_{N_1}$ , then letting $|S_1|\rightarrow 0$ we get that
$$\frac{1}{|S'||K|}\int_{K}\int_{S'}|b(s,t)-m_{ K\times S'}b|dsdt\le \|b\|_{*,N},$$
for any $S'=S_2\times\cdots\times S_{N_1}\subset \T^{N_1-1}.$

Repeating this process for $S_2,\cdots,S_{N_1}$, we obtain that

$$\frac{1}{|K|}\int_{K}|b(s,t)-m_{ K}b|dt\le \|b\|_{*,N}$$
and consequently that $$\sup_{s\in \T^{N_1}}\|b(s,\cdot)\|_{*,N_2}\le \|b\|_{*,N}.$$ The same reasoning leads to $$\sup_{t\in \T^{N_2}}\|b(\cdot,t)\|_{*,N_1}\le \|b\|_{*,N}.$$

For the converse, we write $b(s,t)-m_{S\times K}b$ as follows
$$b(s,t)-m_{S\times K}b=\left(b(s,t)-m_Kb(s)\right)+\left(m_Kb(s)-m_{S\times K}b\right).$$
Hence
\begin{equation}\label{eq:decompomean}
|b(s,t)-m_{S\times K}b|\le |b(s,t)-m_Kb(s)|+|m_Kb(s)-m_{S\times K}b|.
\end{equation}
Integrating both sides of (\ref{eq:decompomean}) over $S\times K$ and with respect to the measure $\frac{dsdt}{|S||K|}$, we obtain
\Beas
L &:=& \frac{1}{|S||K|}\int_S\int_K|b(s,t)-m_{S\times K}b|dsdt\\ &\le& \frac{1}{|S||K|}\int_S\int_K|b(s,t)-m_Kb(s)|dsdt\\ & & + \frac{1}{|S||K|}\int_S\int_K|m_Kb(s)-m_{S\times K}b|dsdt\\ &=& L_1+L_2.
\Eeas
Clearly,
\Beas
L_1 &:=& \frac{1}{|S||K|}\int_S\int_K|b(s,t)-m_Kb(s)|dsdt\\ &\le& \frac{1}{|S|}\int_S\left(\frac{1}{|K|}\int_K|b(s,t)-m_Kb(s)|dt\right)ds\\ &\le&
\frac{1}{|S|}\int_S\|b(s,\cdot)\|_{*,N_2}ds\le C.
\Eeas
On the other hand,
\Beas
L_2 &:=& \frac{1}{|S||K|}\int_S\int_K|m_Kb(s)-m_{S\times K}b|dsdt\\ &=& \frac{1}{|S|}\int_S|m_Kb(s)-m_{S\times K}b|ds\\ &=&
\frac{1}{|S|}\int_S\left|\frac{1}{|K|}\int_K\left(b(s,t)-m_Sb(t)\right)dt\right|ds\\ &\le& \frac{1}{|S||K|}\int_S\int_K|b(s,t)-m_{S}b(t)|dsdt\\ &=&
\frac{1}{|K|}\int_K\left(\frac{1}{|S|}\int_S|b(s,t)-m_{S}b(t)|ds\right)dt\\ &\le& \frac{1}{|K|}\int_K\|b(\cdot,t)\|_{*,N_1}dt\le C.
\Eeas
Thus for any $S\in \T^{N_1}$ and $K\in \T^{N_2}$,
$$\frac{1}{|S||K|}\int_S\int_K|b(s,t)-m_{S\times K}b|dsdt\le 2C.$$
Hence $\|b\|_{*,N}<\infty$. The proof is complete.
\end{proof}
Note that if $C^*$ is the smallest  constant in the equivalent definition above, then $C^*$ is comparable to $\|\,\|_{\bmo(\T^N)}$.

\vskip .2cm
We make the following observation that can be proved exactly as in the one parameter case.
\begin{lem}\label{lem:equivnormbmo}
Let $b\in L^2(\T^N)$. Then
$$\|b\|_{\bmo(\T^N)}\simeq \|b\|_N^*:=\sup_{R\subset \T^N}\inf_{\lambda\in \mathbb C}\frac{1}{|R|}\int_{R}|b(t)-\lambda|dt.$$
\end{lem}

Let us also observe the following.
\begin{lem}\label{lem:averagebmo}
The following assertions hold.
\begin{itemize}
\item[(i)] Given an interval $I$ in $\T$, there is a function in $\BMO(\T)$, denoted $\log_I$  such that
\begin{itemize}
\item the restriction of $\log_I$ to $I$ is $\log\frac{4}{|I|}$.
\item $\|\log_I\|_{\BMO(\T)}\le C$ where $C$ is a constant that does not depend on $I$.
\end{itemize}
\item[(ii)] For any $f_1,\cdots,f_N\in \BMO(\T)$, the function\\ $b(t_1,\cdots,t_N)=\sum_{j=1}^Nf_j(t_j)$ belongs to $\bmo(\T^N)$. Moreover,
$$\|b\|_{\bmo(\T^N)}\le \sum_{j=1}^N\|f_j\|_{\BMO(\T)}.$$
\item[(iii)] There is a constant $C>0$ such that for any $b\in \bmo_m(\T^N)$ and any rectangle $R=I_1\times\cdots\times I_N\subset \T^N$,
\begin{equation}\label{eq:averagebmo}
|m_R b|\le C\left(\log\frac{4}{|I_1|}+\cdots+\log\frac{4}{|I_N|}\right)\|b\|_{\bmo_m(\T^N)},
\end{equation}
and this is sharp.
\end{itemize}
\end{lem}
\begin{proof}
Assertion $(\textrm{ii})$ follows directly from the definition of $\bmo(\T^N)$.

$(\textrm{i})$ is surely well known, we give a proof here for completeness: 
let $J$ be a fixed interval in $\mathbb T$. Let $J_0=J$ and
$J_k$ be the intervals in $\mathbb T$ with the same center
as $J$ and such that $|J_k|=2^k|J|$, here
$k=1,2,\cdots,N-1$ and $N$ is the smallest integer such
that $2^{N}|J|\ge 1$. We define $J_{N}=\mathbb {T}$. Thus,
$$N\le \log_2\frac{4}{|J|}\le N+2.$$
Next, we define $U_0=J_0=J$,
$U_k=J_{k}\setminus J_{k-1}$, for $k=1,\cdots,N$. Now
consider the function $\log_J$ defined on $\mathbb T$ by
\Be\label{test} \log_J(t)=\sum_{k=0}^{N}(N+2-k)\chi_{U_k}(t),\,\,\,t\in \T.\Ee
Clearly, $$\log_J(t)= N+2\simeq \log_{2}\frac{4}{|J|}\,\,\,
\textrm{for all}\,\,\, t\in J.$$
\begin{lem}\label{testdimone}
For each interval $J\subset \mathbb {T}$, the function
$\log_J$ defined by (\ref {test}) belongs to $BMO(\mathbb
T)$.
\end{lem}
\begin{proof} We start by  estimating the $L^2$-norm of $\log_J$. We have
\begin{eqnarray*}||\log_J||_2^2 &=&\sum_{k=0}^{N}(N+2-k)^2|U_k| =
\sum_{k=2}^{N+2}k^2|J_{N+2-k}|\\ &\le& \sum_{k=1}^{N+2}k^2
2^{N+2-k}|J| \le \sum_{k=1}^{N+2}k^2 2^{N+2-k}2^{1-N}\\
&=& 8\sum_{k=1}^{N+2}k^2 2^{-k}.\end{eqnarray*} It is clear
that the last sum in the above equalities is finite and so
$\log_J\in L^2(\mathbb T)$.

For any dyadic interval $I\subset \mathbb {T}$, let $m\in
\{0,\cdots,N+1\}$ be minimal such that $I\cap U_m\neq
\emptyset$, and $l\in \{0,\cdots,N+1\}$ be maximal such that
$I\cap U_{m+l}\neq \emptyset$. Let us estimate
 the length of $I\cap U_j$ for any $m\le j\le m+l$.

If $l=1$ then $I\cap U_m=I$ and there is nothing to say.
If $l=2$ then $|I\cap U_m|\le |I|$ and $|I\cap U_{m+1}|\le
|I|$.

\vskip .2cm
Next, we consider the case $l\ge 2$. We remark that
in this case, half of $U_{m+l-1}$ is contained in $I$. Consequently,
for any $m\le j< m+l$, we have $|I\cap U_j|\le
2\frac{1}{2^{m+l-j-1}}|I|$. Finally, we have $$|I\cap
U_{m+l}|\le 2|I\cap U_{m+l-1}|\le 2|I|.$$ Hence, \Beas
L &:=& \frac{1}{|I|}\int_I|\log_J-(N+2-m-l)|dt\\ &=&
\frac{1}{|I|}\int_I|\sum_{k=m}^{m+l}(m+l-k)\chi_{U_k}|dt\\
&\le& \frac{1}{|I|}\sum_{k=m}^{m+l}(m+l-k)|I\cap U_k|\\
&\le& 4\frac{1}{|I|}\sum_{k=m}^{m+l}(m+l-k)2^{-m-l+k}|I|\\
&=& 4\sum_{k=0}^{k=l}\frac{k}{2^k}\le 6.\Eeas Thus, for
each interval $J\subset \T$, the function $\log_J$ given
by (\ref{test}) belongs to $BMO(\mathbb T)$ and there
exists a positive constant $C$ independent of $J$ such that
$||\log_J||_{BMO(\T)}\le C$.

\vskip .2cm
To prove $(\textrm{iii})$, we observe that by definition, given $b\in \bmo_m(\T^N)$, for any rectangle $S\subset \T^K$, $0<K<N$, $\|m_Sb\|_{*,N-K}$ is uniformly bounded.  It follows from the one parameter estimate of the mean of a function of bounded mean oscillation and the definition of $\bmo_m(\T^N)$ that for any rectangle $Q\subset \T^{N-1}$,
\Beas |m_I(m_Q b)| &\lesssim& \left(\log\frac{4}{|I|}\right)||m_Q b||_{*,1}\\ &\lesssim& \left(\log\frac{4}{|I|}\right)||b||_{\bmo_m(\T^N)}.\Eeas
In particular, for any rectangle $R=I_1\times\cdots\times I_N\subset \T^N$, we have
$$|m_R b|\le C\left(\sum_{j=1}^N\log\frac{4}{|I_j|}\right)\|b\|_{\bmo_m(\T^N)}.$$
The sharpness follows by applying the last inequality to the function\\ $\log_R(t_1,\cdots, t_N)=\sum_{j=1}^N\log_{I_j}(t_j)$, $R=I_1\times\cdots\times I_N$, and using $(\textrm{ii})$.\\ 
The proof is complete.
\end{proof}
Lemma \ref{testdimone} and its proof complete the proof of Lemma \ref{lem:averagebmo}.
\end{proof}
We now reformulate and prove Theorem \ref{bmomultconst}.
\begin{thm}\label{bmomultconst2para}
Let $\phi\in L^2(\T^N)$. Then the following assertions are equivalent.
\begin{itemize}
\item[(a)] $\phi$ is multiplier of $bmo(\T^N)$.
\item[(b)] $\phi$ is a constant.
\end{itemize}
\end{thm}
\begin{proof}
Clearly, $(\textrm{b})\Rightarrow (\textrm{a})$. We prove that $(\textrm{a})\Rightarrow (\textrm{b})$.

Assume that $\phi\in L^2(\T^N)$ is a multiplier of $\bmo(\T^N)$. Then for any $f\in \bmo(\T^N)$, and any integer $0<N_1<N$, $N_2=N-N_1$, $||(\phi f)(.,t)||_{*,N_1}$ is uniformly bounded for all $t\in \T^{N_2}$ fixed and $\|(\phi f)(.,t)\|_{*,N_1}\le \|\phi f\|_{\bmo(\T^N)}$. Let us take as $f$ the function $f(s,t)=\log_R(s,t)=\sum_{k=1}^{N_1}\log_{S_k}(s)+\sum_{j=1}^{N_2}\log_{Q_j}(t)$, $R=S\times Q\subset \T^{N_1}\times \T^{N_2}$, $S=S_1\times\cdots\times S_{N_1}$, $Q=Q_1\times\cdots\times Q_{N_2}$, $S_k\subset \T,\,\,\,Q_j\subset \T$.
Then it follows that
$$\frac{1}{|S|}\int_{S}|\phi(s,t)f(s,t)-m_S\left(\phi f\right)|ds\le \|\phi f\|_{\bmo(\T^N)}, \textrm{for all}\,\,\, S\subset \T^{N_1}.$$
But from assertion $(\textrm{i})$ of Lemma \ref{lem:averagebmo} we have that for any $t\in Q\subset \T^{N_2}$ fixed,

\Beas
L &:=& \frac{\sum_{j=1}^{N_2}\log\frac{4}{|Q_j|}+\sum_{k=1}^{N_1}\log\frac{4}{|S_k|}}{|S|}\int_{S}|\phi(s,t)-m_S\phi|ds\\ &\lesssim& 
\frac{1}{|S|}\int_{S}|\phi(s,t)\log_R(s,t)-m_S(\phi\log_R)|ds\\ &\le& ||\phi\log_R||_{\bmo(\T^N)}\\ &\lesssim& \|M_\phi\|,
\Eeas
where $\|M_\phi\|$ is the norm of the multiplication by $\phi$, $M_\phi (f)=\phi f$.

Hence for any $S\subset \T^{N_1},\,\,\,Q\subset \T^{N_2}$ and $t\in Q$,
\Be\label{eq:constequat1}\left(\sum_{j=1}^{N_2}\log\frac{4}{|Q_j|}+\sum_{k=1}^{N_1}\log\frac{4}{|S_k|}\right)\left(\frac{1}{|S|}\int_{S}|\phi(s,t)-m_S\phi|ds\right)<\infty.
\Ee

Letting for example $|Q_1|\rightarrow 0$ in (\ref{eq:constequat1}), we see that necessarily,\\ $\phi(s,t)=\phi(t)$ for any $s\in S\subset \T^{N_1}$. As $N_1$ runs through $(0,N)$, we obtain\\ that for any $(t_1,\cdots,t_N)\in \T^N$,
$$\phi(t_1,\cdots,t_N)=\phi(t_j)=\phi(t_{j_1},\cdots,t_{j_k}),\,\,\,j,j_l\in \{1,\cdots,N\},$$
$0<k<N.$
The latter gives that $\phi$ is a constant.
\end{proof}
We have the following consequence which says that the only bounded functions in $\lmo(\T^N)$ are the constants. This is pretty different from the one parameter case (\cite{stegenga}).
\begin{cor}\label{cor:cormain}
Assume that $\phi\in L^\infty(\T^N)$ and
\Be\label{charact01}
||\phi||_{*,\log,N}:=\sup_{R=I_1\times \cdots\times I_N\subset \T^N}\frac{\sum_{j=1}^N\log\frac{4}{|I_j|}}{|R|}\int_{R}|\phi(t)-m_R\phi|dt<\infty.
\Ee
Then $\phi$ is a constant.
\end{cor}
\begin{proof} Following Theorem \ref{bmomultconst} we only need to prove that any bounded function $\phi$ which satisfies (\ref{charact01}) is a multiplier of $\bmo(\T^N)$. For this
we first recall that if $f\in \bmo(\T^N)$, then for any rectangle $R=I_1\times\cdots\times I_N\subset \T^N$, we have the estimate

\begin{equation*}
|m_R f|\lesssim \left(\log\frac{4}{|I_1|}+\cdots+\log\frac{4}{|I_N|}\right)||f||_{\bmo(\T^N)}.
\end{equation*}
Now assume that $\phi\in L^\infty(\T^N)$ and satisfies (\ref{charact01}), and let $f\in \bmo(\T^N)$. Then using the above estimate, we obtain for any $R=I_1\times\cdots\times I_N\subset \T^N$,
\Beas
&&\frac{1}{|R|}\int_{R}|(f\phi)(t)-m_R\phi m_Rf|dt\\ &\le& \frac{1}{|R|}\int_{R}|\phi(t)||f(t)-m_Rf|dt+\\ & &
\frac{1}{|R|}\int_{R}|m_Rf||\phi(t)-m_R\phi|dt\\ &\le& \frac{||\phi||_{L^\infty(\T^N)}}{|R|}\int_{R}|f(t)-m_Rf|dt+\\ & & \frac{||f||_{\bmo(\T^N)}\left(\log\frac{4}{|I_1|}+\cdots+\log\frac{4}{|I_N|}\right)}{|R|}\int_{R}|\phi(t)-m_R\phi|dt\\ &\le& \left(||\phi||_{L^\infty(\T^N)}+||\phi||_{*,\log,N}\right)||f||_{\bmo(\T^N)}.
\Eeas
It follows from the latter and Lemma \ref{lem:equivnormbmo} that if $\phi$ is bounded and satisfies (\ref{charact01}), then for any $f\in \bmo(\T^N)$, $\phi f$ belongs to $\bmo(\T^N)$. That is $\phi$ is a multiplier of $\bmo(\T^N)$.
The proof is complete.
\end{proof}

\begin{rem}\label{remark2} Let us first recall that in the one parameter case, it is a result of D. Stegenga \cite{stegenga} that $L^\infty(\T)\cap \LMO(\T)$ is the exact range of pointwise multipliers of $\BMO(\T)$. Let us define another little $\LMO$ space in the two-parameter case $\lmo_{inv}(\T^2)$ as follows.
\begin{defin}
A function $b\in L^2(\T^2)$ is in $\lmo_{inv}(\T^2)$ if there is a constant $C>0$ such that
$\|b(\cdot,t)\|_{*,\log,1}\le C$ for all $t\in \T$ and $\|b(s,\cdot)\|_{*,\log,1}\le C$ for all $s\in \T$.
\end{defin}

Clearly, $\lmo_{inv}(\T^2)$ is a subspace of $\lmo_m(\T^2)$. The one parameter intuition and the equivalent definition of $\bmo(\T^2)$ in Proposition \ref{prop:equivdefbmo} may lead one to claim that any function $\phi\in L^\infty(\T^2)\cap \lmo_{inv}(\T^2)$ is a multiplier of $\bmo(\T^2)$. This is not the case as the above results show and since $L^\infty(\T^2)\cap \lmo_{inv}(\T^2)$ contains more than constants. For example, for any $\phi_1,\phi_2\in L^\infty(\T)\cap \LMO(\T)$, the function $\phi:\,\, (s,t)\mapsto \phi_1(s)\phi_2(t)$ belongs to $L^\infty(\T^2)\cap \lmo_{inv}(\T^2)$.
\end{rem}
\subsection{Proof of  Theorem \ref{meanbmomulti}}
We prove Theorem \ref{meanbmomulti} in this section.

\begin{proof}[{\bf Proof of Theorem \ref{meanbmomulti}}]
$(\textrm{i})\Rightarrow (\textrm{ii})$: we start by proving that any multiplier from $\bmo(\T^N)$ to $\bmo_m(\T^N)$ is a bounded function. We  recall the following estimate of the mean over a rectangle of functions in $\bmo_m(\T^N)$:
$$|m_Rb|\lesssim \left(\log\frac{4}{|I_1|}+\cdots +\log\frac{4}{|I_N|}\right)||b||_{\bmo_m(\T^N)},\,\,\,R=I_1\times\cdots\times I_N\subset \T^N.$$
It follows that if $\phi$ is multiplier from $\bmo(\T^N)$ to $\bmo_m(\T^N)$, then for any $b\in \bmo(\T^N)$ and for any rectangle $R=I_1\times\cdots\times I_N\subset \T^N$,

\begin{align}\label{eq:multiboundedd}
|m_R\left(b\phi\right)|
&\lesssim \left(\sum_{j=1}^N\log\frac{4}{|I_j|}\right)||b\phi||_{\bmo_m(\T^N)}\\
&\le C\left(\sum_{j=1}^N\log\frac{4}{|I_j|}\right)\|M_\phi\|_{\bmo(\T^N)\rightarrow \bmo_m(\T^N)}\|b\|_{\bmo(\T^N)}.\notag
\end{align}

Applying (\ref{eq:multiboundedd}) to $b=\log_{I_1}+\cdots+\log_{I_N}$ and using assertions $(\textrm{i})$ and $(\textrm{ii})$ of Lemma \ref{lem:averagebmo}, we see  that there is a constant $C>0$ such that
$$|m_R\phi|\le C\|M_\phi\|_{\bmo(\T^N)\rightarrow \bmo_m(\T^N)},\,\,\,\textrm{for any}\,\,\, R=I_1\times\cdots\times I_N\subset \T^N.$$
We conclude that $\phi\in L^\infty(\T^N)$.

\vskip .2cm
To prove that $\phi\in \lmo_m(\T^N)$, we only need by the definition of $\lmo_m(\T^N)$ to check that for any integer $0<M<N$, any rectangle $R\subset \T^{M}$, $\|m_R\phi\|_{*,\log,K}$ ($K=N-M$) is uniformly bounded. Let $S$ be a rectangle in $\T^{K}$ and $\log_S(t_1,\cdots,t_K)=\log_{S_1}(t_1)+\cdots+\log_{S_K}(t_K)$, $S_j\subset \T$ be again the associated sum of functions which are uniformly in $\BMO(\T)$. We have
\Beas
L &:=& \frac{\sum_{j=1}^K\log\frac{4}{|S_j|}}{|S_j|}\int_{S}|m_R \phi(t)-m_{S\times R}\phi|dt\\ &=& \frac{1}{|S|}\int_{S}|m_R(\phi \log_S)(t)-m_{S\times R}(\phi \log_S)|dt\\  &\le&  ||m_R(\phi \log_S)||_{\bmo_m(\T^K)}\\ 
&\le& ||\phi \log_S||_{\bmo_m(\T^N)}\\ &\lesssim&
||M_\phi||_{\bmo(\T^N)\rightarrow \bmo_m(\T^N)}||\log_S||_{\bmo(\T^N)}\\ &=& ||M_\phi||_{\bmo(\T^N)\rightarrow \bmo_m(\T^N)}||\log_S||_{\bmo(\T^K)}\\ &\lesssim&
||M_\phi||_{\bmo(\T^N)\rightarrow \bmo_m(\T^N)}.
\Eeas
Hence for any integer $0<M<N$ and for any $R\subset \T^{M}$, $\|m_R\phi \|_{*,\log, N-M}$ is uniformly bounded. Thus, by definition, $\phi\in \lmo_m(\T^N)$.

\vskip .2cm

$(\textrm{ii})\Rightarrow (\textrm{i})$: Let $\phi\in L^\infty(\T^N)\cap \lmo_m(\T^N)$. To prove that\\ $\phi\in \mathcal {M}(\bmo(\T^N),\bmo_m(\T^N))$, we only need to check that
for any integer $0<M<N$, for any rectangle $R\subset \T^{M}$, and any $f\in \bmo(\T^N)$, $\|m_R(\phi f)\|_{*,K}$ ($K=N-M$) is uniformly bounded. Let $S$ be a rectangle in $\T^{K}$. Then
\Beas
L &:=& \frac{1}{|S|}\int_{S}|m_R (\phi f)(t)-m_{S\times R}\phi m_{S\times R}f|dt\\ &\le& \frac{1}{|S|}\int_{S}|m_R\left[(\phi-m_{S\times R}\phi) (f-m_{S\times R}f)\right](t)|dt\\ &+& \frac{1}{|S|}\int_{S}|\left(m_{S\times R}f\right)(m_R\phi)(t)-m_{S\times R}\phi m_{S\times R}f|dt\\ &+&  \frac{1}{|S|}\int_{S}|\left(m_{S\times R}\phi\right)(m_Rf)(t)-m_{S\times R}\phi m_{S\times R}f|dt\\ &=& L_1+L_2+L_3.
\Eeas
To estimate the first term, we only use that $\phi\in L^\infty(\T^N)$ to obtain
\Beas
&&L_1 := \frac{1}{|S|}\int_{S}|m_R\left[(\phi-m_{S\times R}\phi) (f-m_{S\times R}f)\right](t)|dt\\ &\le& \frac{1}{|S||R|}\int_{S\times R}|\left[(\phi-m_{S\times R}\phi) (f-m_{S\times R}f)\right](t_1,\cdots,t_N)|dt_1\cdots dt_N\\  &\le& \frac{||\phi||_{L^\infty(\T^N)}}{|S||R|}\int_{S\times R}|f(t_1,\cdots,t_N)-m_{S\times R}f|dt_1\cdots dt_N\\ &\le& ||\phi||_{L^\infty(\T^N)}||f||_{\bmo(\T^N)}.
\Eeas

For the second term, we use the fact that as $\|m_Rf\|_{*,K}$ is uniformly bounded,  \Beas |m_{S\times R}f|=|m_S(m_Rf)| &\lesssim& \left(\sum_{j=1}^{K}\log\frac{4}{|S_j|}\right)||m_Rf||_{*,K}\\ &\le& \left(\sum_{j=1}^{K}\log\frac{4}{|S_j|}\right)||f||_{\bmo_m(\T^N)},\Eeas
$S=S_1\times\cdots\times S_{K}\subset \T^{K}, K=N-M$.
Consequently,
\Beas
L_2 &:=& \frac{1}{|S|}\int_{S}|\left(m_{S\times R}f\right)(m_R\phi)(t)-m_{S\times R}\phi m_{S\times R}f|dt\\ &\lesssim& \frac{\left(\sum_{j=1}^K\log\frac{4}{|S_j|}\right)||f||_{\bmo(\T^N)}}{|S|}\int_{S}|(m_R\phi)(t)-m_{S\times R}\phi|dt\\ &\le& ||f||_{\bmo_m(\T^N)}||m_R\phi||_{*,\log,K}\\ &\le& ||f||_{\bmo_m(\T^N)}||\phi||_{\lmo_m(\T^N)}.
\Eeas
The last term only uses the fact that $\phi\in L^\infty(\T^N)$.
\Beas L_3 &:=& \frac{1}{|S|}\int_{S}|\left(m_{S\times R}\phi\right)(m_Rf)(t)-\left(m_{S\times R}\phi\right)\left( m_{S\times R}f\right)|dt\\ &\le&
\|\phi\|_{L^\infty(\T^N)}\frac{1}{|S|}\int_{S}|(m_Rf)(t)- m_{S\times R}f|dt\\ &\le& \|\phi\|_{L^\infty(\T^N)}\|m_Rf\|_{*,K}\\ &\le&
||\phi||_{L^\infty(\T^N)}||f||_{\bmo_m(\T^N)}.
\Eeas
The estimates of $L_1$, $L_2$ and $L_3$, and Lemma \ref{lem:equivnormbmo} allow to conclude that
$$||\phi f||_{\bmo_m(\T^N)}\lesssim \left(||\phi||_{L^\infty(\T^N)}+||\phi||_{\lmo_m(\T^N)}\right)||f||_{\bmo(\T^N)}.$$
This complete the proof of the theorem.
\end{proof}

\section{Multipliers to $\bmo(\T^N)$}
We would like to deduce some consequences of the above approach. We consider multipliers from any Banach space of functions on $\T^N$ (strictly) containing $\bmo(\T^N)$ to $\bmo(\T^N)$. We have the following general result.
\begin{thm}\label{Xtobmo}
Let $X$ be any Banach space of functions on $\T^N$ that strictly contains $\bmo(\T^N)$. Then $\mathcal {M}(X,\bmo(\T^N))=\{0\}$.
\end{thm}
\begin{proof}
Clearly, $0$ sends any function of $X$ to $\bmo(\T^N)$ by multiplication. Now let $\phi$ be any multiplier from $X$ to $\bmo(\T^N)$, then $\phi$
is also a multiplier from $\bmo(\T^N)$ to itself. It follows from Theorem \ref{bmomultconst} that $\phi$ is a constant $C$. Suppose that $C\neq 0$ and recall that $\bmo(\T^N)$ is a proper subspace of $X$. Then for any $f\in X$, we have that $f=C(\frac{1}{C}f)=\phi(\frac{1}{C}f)\in \bmo(\T^N)$. This contradicts the fact that $\bmo(\T^N)$ is a strict subspace of $X$. Hence $C$ is necessarily $0$. The proof is complete.
\end{proof}
Taking as $X$, the Chang-Fefferman $\BMO$ space or $\bmo_m(\T^N)$ we have as corollary the following.
\begin{cor}
We have $$\mathcal {M}(\BMO(\T^N),\bmo(\T^N))=\mathcal {M}(\bmo_m(\T^N),\bmo(\T^N))=\{0\}.$$
\end{cor}

The author would like to thank the referee for comments and observations that improved the presentation of this note.


\end{document}